\documentclass[leqno]{article}
\usepackage{beamerarticle} 
\usepackage{geometry}      \usepackage{amsthm}        

\usepackage[T2A]{fontenc}
\usepackage[utf8]{inputenc}
\usepackage[english,russian]{babel}

\ifdefined\isenglish\else\def\isenglish{1}\fi

\if\isenglish1
  \AtBeginDocument{\selectlanguage{english}}
\else
  \AtBeginDocument{\selectlanguage{russian}}
\fi

\ifdefined\whichmode\else\def\whichmode{1}\fi

\newcommand{\mytitle}{Pointed Gromov--Hausdorff Distance between $l^p$-Wedge Sums and Isometric Embeddings of $l^{\infty}$}

\newcommand{\myabstract}{
The paper studies $l^{p}$-wedge sums of metric spaces and provides a two-sided estimate for the pointed Gromov--Hausdorff distance between them. Using an upper bound for the Gromov–Hausdorff distance, the existence of an isometric embedding of $l^{\infty}$ into the Gromov--Hausdorff space $\cGH$ of all non-empty metric spaces is proved.

Bibliography: 15 titles.
}

\newcommand{\myname}{E.~Abdullaev}

\newcommand{\mydate}{\ifnum\day<10 0\fi\the\day.\ifnum\month<10 0\fi\the\month.\the\year}

\usepackage{mathtools} \usepackage{amsfonts,amssymb,mathrsfs,amscd}
\usepackage{stmaryrd}

\usepackage{graphicx}
\usepackage{verbatim}
\usepackage{url}
\usepackage{alltt}
\usepackage{indentfirst}
\usepackage{qrcode}

\theoremstyle{plain}

\newtheorem{thm}{Theorem}[section]
\newtheorem{prop}[thm]{Proposition}
\newtheorem{cor}[thm]{Corollary}
\newtheorem{ass}[thm]{Assertion}
\newtheorem{lem}[thm]{Lemma}

\theoremstyle{definition}

\newtheorem{dfn}[thm]{Definition}
\newtheorem{rk}[thm]{Remark}

\newtheorem*{notation}{Notation}

\renewcommand{\a}{\alpha}
\newcommand{\bt}{\beta}
\newcommand{\dl}{\delta}
\newcommand{\D}{\Delta}
\newcommand{\g}{\gamma}
\newcommand{\e}{\varepsilon}
\newcommand{\vk}{\varkappa}

\newcommand{\cGH}{\mathcal{GH}}

\newcommand{\cR}{\mathcal{R}}
\newcommand{\cM}{\mathcal{M}}

\newcommand{\cP}{\mathcal{P}}

\newcommand{\N}{\mathbb{N}}

\newcommand{\R}{\mathbb{R}}

\newcommand{\diam}{{\operatorname{diam}}}
\newcommand{\Rad}{{\operatorname{Rad}}}
\newcommand{\dis}{{\operatorname{dis}}}
\newcommand{\codis}{{\operatorname{codis}}}

\renewcommand{\:}{\colon}
\newcommand{\rom}[1]{{\em #1}}

\makeatletter
\@ifundefined{RBibitem}{
\def\RBibitem#1{\normalfont\bibitem{#1}\ignorespaces}

\def\by{\unskip\normalfont~}
  \def\paper{\unskip\normalfont. }
  \def\jour{\unskip\normalfont, \itshape }
  \def\publaddr{\unskip\normalfont, }
  \def\publ{\unskip\normalfont, }
  \def\pages{\unskip\normalfont, pp.~}

\def\inbook{\unskip\normalfont\ \itshape }
  \def\eds{\unskip\normalfont, ed. by }
  \def\book{\unskip\normalfont. \itshape }
  \def\issue{\unskip\normalfont, no.~}

\def\yr{\unskip, \normalfont}
  \def\vol{\unskip, \normalfont\bfseries}

  \def\bysame{\by\hbox to3em{\hrulefill}\thinspace\kern\z@}
  \def\preprint{\unskip\normalfont, preprint }

  \def\finalinfo{\unskip\normalfont, }
  \def\@@finalinfo{\finalinfo}

\def\morerref{\unskip\spacefactor=1000 \newblock}

}{}
\makeatother

\usepackage{hyperref}      
\begin{document}
\title{\mytitle}
\author{\myname}
\date{\today}

\maketitle

\begin{abstract}
\myabstract
\end{abstract}

\ifnum\whichmode=1
\tableofcontents
\fi

\section*{Introduction}

The concept of distance between metric spaces, introduced by F.~Hausdorff~\cite{Hausdorff} and generalized in the works of D.~Edwards~\cite{Edwards} and M.~Gromov~\cite{Gromov}, is one of the fundamental tools of modern metric geometry.
On the class of all metric compact spaces, the Gromov-Hausdorff distance is a metric that turns this class into a complete separable geodesic metric space.
The fundamental properties of this distance and the geometry of the Gromov-Hausdorff space are being actively studied (see, e.g.,~\cite{BurBurIva, IvaNikTuz, NP01, BogatyTuzhilin, IvaTuzSimplex}).

Despite being deeply studied theoretically, the exact computation of the Gromov-Hausdorff distance in explicit form is an extremely difficult problem.
The complexity is due to the necessity of considering all possible correspondences between the spaces.
From a computational point of view, the problem of finding the Gromov-Hausdorff distance for finite metric spaces belongs to the class of $NP$-hard problems~\cite{NP01, NP02}.
To date, the exact values of the distance are known only for very narrow classes of spaces, such as segments and circles, spheres of fixed dimensions, as well as discrete metric spaces of a special type (metric simplices)~\cite{IvaTuzSimplex, JiTuzhilin, GrigIvaTuz}.

In view of these difficulties, finding effective two-sided bounds for the Gromov-Hausdorff distance between spaces with a specific geometric structure is an important problem.
Of particular interest is the study of composite spaces obtained using standard set-theoretic and topological operations.
In~\cite{Me_GH_Product_2026} the author investigated $l^p$-products of metric spaces and obtained bounds for the Gromov-Hausdorff distance between them.
A natural development of this research direction is the consideration of $l^p$-wedge sums (one-point unions) of pointed metric spaces, where the metric is consistent with the geometry of the spaces relative to the chosen base points.

It is worth noting that the distance between pointed spaces (spaces with a base point) has various modifications depending on the conditions imposed on the preservation of the base points.
In this paper, we use an approach in which correspondences are required to link the base points; such a definition is consistent with the general concept of the Gromov-Hausdorff distance for $F$-spaces introduced in~\cite{Me_DGH_2026}.

This paper is devoted to obtaining two-sided bounds for the pointed Gromov-Hausdorff distance between $l^p$-wedge sums of metric spaces. The paper is organized as follows.
Section 1 provides the basic definitions and auxiliary statements.
Section 2 considers the general case: Theorem~\ref{thm:GH_upper_bound} proves an upper bound for the distance between arbitrary $l^p$-wedge sums in terms of the distances between their components, and Theorem~\ref{thm:GH_lower_bound} establishes a corresponding lower bound depending on the diameters of the components.
Finally, Section 3 investigates an important special case where the compared wedge sums are constructed from the same family of unit radius spaces by scaling (stretching) their components using different sets of non-negative coefficients.
In Theorem~\ref{thm:linear} an upper bound for the Gromov-Hausdorff distance between such spaces is obtained, expressed in terms of these coefficients and the diameters of the original spaces.

Section 4 is devoted to another problem --- the isometric embedding of metric spaces into the Gromov--Hausdorff space $\cGH$ of all non-empty metric spaces (unlike the previous sections, here we consider ordinary metric spaces rather than pointed ones). The question regarding the possibility of such an embedding has its own history. It is well known that any bounded metric space can be isometrically embedded into $\cGH$; this classical result belongs to A.~O.~Ivanov and A.~A.~Tuzhilin~\cite{IvaTuz22}. On the other hand, a recent preprint~\cite{Fukano} considered embeddings of an arbitrary compact metric space into the space of compact metric spaces. The present paper does not overlap with these works and is devoted to a fundamentally different problem --- the embedding of the unbounded space $l^{\infty}$ into $\cGH$. Using the theory developed in Section 2, we construct an explicit isometric embedding of the auxiliary space $\cP_{00}(\R_{+})$ of all non-empty subsets of $\R_{+}$, containing 0, into $\cGH$. Since there exists an embedding $l^{\infty} \to \cP_{00}(\R_{+})$, which we also establish, this guarantees the existence of the desired isometric embedding of $l^{\infty}$ into $\cGH$.

\subsection*{Acknowledgments}
\begin{sloppypar}
The author expresses gratitude to his scientific advisor, Doctor of Physical and Mathematical Sciences, Professor \mbox{A.~A.~Tuzhilin}, and Doctor of Physical and Mathematical Sciences, Professor \mbox{A.~O.~Ivanov}, for formulating the problem and their constant attention to the work.
\end{sloppypar}

\section{Main definitions and preliminary results}

\subsection{Gromov--Hausdorff distance for metric spaces}

Let $\cGH$ be the class of all nonempty metric spaces, and $\cM \subset \cGH$ be the class of all nonempty metric compacts. For arbitrary $X, Y \in \cGH$ let $d_X$ and $d_Y$ denote their metrics.

Let $R \subset X \times Y$ be an arbitrary nonempty subset.
Define the \emph{distortion}
$$\dis(R) = \sup_{\substack{(x, y) \in R \\ (\tilde{x}, \tilde{y}) \in R}} \bigl| d_{X}(x, \tilde{x}) - d_{Y}(y, \tilde{y}) \bigr|.$$
In particular, for an arbitrary map $f \: X \to Y$ we define its distortion as the distortion of its graph:
$$\dis(f) = \sup_{x, \tilde{x} \in X} \Bigl| d_{X}(x, \tilde{x}) - d_{Y}\bigl(f(x), f(\tilde{x})\bigr)\Bigr|.$$

For a pair of maps $f_1 \: X \to Y, \ f_2 \: Y \to X$ we define their \emph{codistortion}
$$\codis(f_1, f_2) = \sup_{\substack{x \in X \\ y \in Y}} \Bigl| d_{X}\bigl(x, f_2(y)\bigr) - d_{Y}\bigl(y, f_1(x)\bigr)\Bigr|.$$

A \emph{correspondence} between $X$ and $Y$ is a subset $R \subset X \times Y$ whose projections onto $X$ and $Y$ are surjective. The set of all correspondences between $X$ and $Y$ will be denoted by $\cR(X, Y)$.
The Gromov--Hausdorff distance can be defined as
$$d_{GH}(X, Y) = \frac{1}{2} \inf_{R \in \cR(X, Y)} \dis(R).$$

In what follows, when discussing the properties of the Gromov--Hausdorff distance, we follow \cite{BurBurIva}, \cite{BogatyTuzhilin}.
The following proposition provides an equivalent definition of the Gromov--Hausdorff distance.

\begin{prop}
For any $X, Y \in \cGH$ it holds that
$$d_{GH}(X, Y) = \frac{1}{2} \inf_{\substack{f_1 \: X \to Y \\ f_2 \: Y \to X}} \max \bigl\{ \dis(f_1), \dis(f_2), \codis(f_1, f_2) \bigr\},$$
where the infimum is taken over all maps $f_1, f_2$.
\end{prop}

For any metric space $X$, let us define its \emph{diameter}:
$$\diam(X) = \sup_{x, y \in X} d_{X}(x, y).$$

\begin{prop}
For any bounded metric spaces $X, Y$, the following holds\/\rom:
$$\bigl| \diam(X) - \diam(Y) \bigr| \le 2d_{GH}(X, Y) \le \max\bigl\{ \diam(X), \ \diam(Y) \bigr\}.$$
\end{prop}

\subsection{Gromov--Hausdorff distance between pointed metric spaces}

\begin{dfn}
We call a metric space $(X, d_{X})$ with a chosen base point $x_0$\emph{ a pointed metric space}, and we denote it by $(X, d_{X}, x_0)$ or $(X, x_0)$.
\end{dfn}

In~\cite{Me_DGH_2026} the concept of an $F$-space was introduced as a set with a given family of generalized pseudometrics and base points, and the Gromov--Hausdorff distance for $F$-spaces was defined. The following definition of the Gromov--Hausdorff distance is consistent with that definition and is a special case of it.

\begin{dfn}
Let $(X, x_0), (Y, y_0)$ be pointed metric spaces. We then introduce the \emph{pointed Gromov--Hausdorff distance}
$$d_{pGH}\bigl((X,x_0),(Y,y_0)\bigr)=\frac{1}{2} \inf_{R \in \cR \left((X,x_0),(Y,y_0)\right)} \dis(R), $$ where
$$\cR ((X, x_0), (Y, y_0)) = \Bigl\{ R \in \cR(X,Y) \ \big| \ (x_0,y_0) \in R \Bigr\}.$$
\end{dfn}

For any pointed metric space $(X, x_0)$, we define its \emph{radius}:
$$\Rad(X, x_0) = \sup_{x \in X} d_{X}(x, x_0).$$

\begin{prop}[\cite{Me_DGH_2026}]
For any bounded pointed metric spaces $(X, x_0), (Y, y_0)$ the following holds\/\rom:
$$\bigl| \Rad(X, x_0) - \Rad(Y, y_0) \bigr| \le 2d_{pGH}\bigl((X, x_0), (Y, y_0)\bigr) \le \max\bigl\{ \diam(X), \ \diam(Y) \bigr\}.$$
\end{prop}

\subsection{Products and wedge sums of metric spaces}

\begin{dfn} Let $\{(X_{n}, x_n, d_{X_{n}})\}_{n=1}^{\infty}$ be a sequence of pointed metric spaces, such that
$$\sum_{n=1}^{\infty}(\diam(X_n))^{p} < \infty,$$
where $1 \le p < \infty$. Then on the Cartesian product $\prod_{n=1}^{\infty} X_{n}$ we introduce the metric
$$ d_{X}\bigl(\{x_n\}_{n=1}^{\infty}, \{y_n\}_{n=1}^{\infty} \bigr) = \Biggl(\sum_{n = 1}^{\infty} \bigl(d_{X_{n}}(x_{n}, y_{n}) \bigr)^{p} \Biggr)^{1/p},$$
we also set $x_0 = \{x_{n}\}_{n=1}^{\infty}$ and denote the obtained pointed metric space $(X, d_{X}, x_{0})$ by $(l^{p})\prod_{n} X_{n}$.
\end{dfn}

\begin{dfn}
Let $\{(X_{\a}, x_{\a})\}_{\a\in A}$ be an arbitrary family of pointed metric spaces, such that $\sup_{\alpha \in A} \diam(X_{\alpha})<\infty$.
Then on the Cartesian product $\prod_{\a \in A} X_{\a}$ we introduce the metric
$$ d_{X}\bigl(\{x_{\alpha}\}_{\alpha \in A}, \{y_{\alpha}\}_{\alpha \in A} \bigr) = \sup_{\alpha \in A} d_{X_{\a}}(x_{\alpha}, y_{\alpha}),$$
we also set $x_0 = \{x_{\a}\}_{\a \in A}$ and denote the obtained pointed metric space $(X, d_{X}, x_{0})$ by $(l^{\infty})\prod_{\alpha \in A} X_{\alpha}$.
\end{dfn}

\begin{dfn}
Let $\{(X_{\a},x_{\a})\}_{\a \in A}$, $\{(Y_{\a},y_{\a})\}_{\a \in A}$ be families of pointed metric spaces.
Let
$$X = \prod_{\a \in A} X_{\a}, \quad Y = \prod_{\a \in A} Y_{\a},$$
where the metrics on $X$ and $Y$ are chosen arbitrarily.
For given correspondences $R_{\a} \in \cR(X_{\a}, Y_{\a})$
consider their product (as a product of sets)
$$R = \prod_{\a \in A} R_{\a}$$
and consider it as a correspondence $R \in \cR(X, Y).$
We will call the obtained correspondence the \emph{product of correspondences}.
If we want to emphasize that the correspondence is considered specifically in the $l^p$-product, we will write
$$(l^p) \prod_{\a \in A} R_{\a} \in \cR\biggl( (l^p) \prod_{\a \in A} (X_{\a}, x_{\a}), \ (l^p) \prod_{\a \in A} (Y_{\a}, y_{\a}) \biggr).$$
\end{dfn}

\begin{dfn}
Let $\{(X_{\a}, d_{X_{\a}}, x_{\a})\}_{\a \in A}$ be a family of pointed metric spaces.
For any $1 \le p \le \infty$ we define their \emph{wedge sum}
$$X = (l^p) \bigvee_{\a \in A} (X_{\a},x_{\a}),$$
in which all points $x_{\a}$ are identified and on which the metric is defined as follows:
$$d_{X}(x,y ) =
\begin{cases}
d_{X_{\a}}(x, y), & \text{ if } x, y \in X_{\a}, \ \a \in A, \\
\Bigl(\bigl(d_{X_{\a}}(x,x_{\a})\bigr)^p+\bigl(d_{X_{\bt}}(y,x_{\bt})\bigr)^p\Bigr)^{1/p}, & \text{ if } x \in X_{\a}, \ y \in X_{\bt}, \ \a \ne \bt, \ p \ne \infty, \\
\max\bigl\{ d_{X_{\a}}(x, x_{\a}), \ d_{X_{\bt}}(y, x_{\bt})\bigr\}, & \text{ if } x \in X_{\a}, \ y \in X_{\bt}, \ \a \ne \bt, \ p = \infty.
\end{cases}
$$
\end{dfn}

\begin{rk}
In the previous definition, there is no need to verify the triangle inequality for the wedge sum, since it is satisfied automatically.
Indeed, let
$$x, y, z \in X = (l^{p}) \bigvee_{\a \in A} (X_{\a}, x_{\a}),$$ then for some $\alpha, \beta, \gamma$ it must be that $x \in X_{\alpha}, \ y \in X_{\beta}, \ z \in X_{\gamma}$, and then the triangle inequality follows from the fact that the wedge sum can be considered as a subspace of the product:
$$(l^{p}) \bigvee_{\dl \in \{\alpha, \beta, \gamma\}} (X_{\dl}, x_{\dl}) \subset (l^{p}) \prod_{\dl \in \{\alpha, \beta, \gamma\}} (X_{\dl}, x_{\dl}).$$
\end{rk}

\begin{dfn}
For arbitrary families of pointed metric spaces $\{(X_{\a},x_{\a})\}_{\a \in A}$ and $\{(Y_{\a},y_{\a})\}_{\a \in A}$ and for given correspondences $R_{\a}\in \cR((X_{\a}, x_{\a}), (Y_{\a},y_{\a}))$ we define the \emph{wedge sum of correspondences} by analogy with the product of correspondences:
$$R=\bigvee_{\a \in A}R_{\a}:=\bigcup_{\a \in A}R_{\a} \in \cR \biggl( \bigvee_{\a \in A} (X_{\a}, x_{\a}), \bigvee_{\a \in A}(Y_{\a}, y_{\a}) \biggr).$$
To emphasize that the correspondence is considered specifically as a correspondence in the $l^p$-wedge sum, we will write
$$R = (l^{p})\bigvee_{\a \in A} R_{\a} \in \cR \biggl( (l^{p})\bigvee_{\a \in A} (X_{\a}, x_{\a}), (l^{p}) \bigvee_{\a \in A}(Y_{\a}, y_{\a}) \biggr).$$
\end{dfn}

\begin{notation}
For an arbitrary cardinal number $\kappa$ let $\D_{\kappa}$ denote the metric space of cardinality $\#\D_{\kappa}=\kappa$ in which all non-zero distances are equal to $1$.
\end{notation}

\subsection{Hausdorff Distance}

\begin{dfn}[\cite{BurBurIva}]
Let $(X, d_{X})$ be a metric space.
On the set $2^{X} \setminus \{\emptyset\}$, we introduce the \emph{Hausdorff distance}
$$d_{X}^{H}(A, B) = \max \biggl\{ \sup_{a \in A} \inf_{b \in B} d_{X}(a, b), \ \sup_{b \in B} \inf_{a \in A} d_{X}(a, b) \biggr\}.$$
\end{dfn}

The Hausdorff distance is equivalently reformulated via correspondences as follows.

\begin{rk}
Note that the Hausdorff distance between sets is given by the formula:
$$ d^{H}_{X}(A, B) = \max \bigl\{\delta(A, B), \, \delta(B, A) \bigr\}, $$
where
$$ \delta(A, B) := \sup_{a \in A} \inf_{b \in B} d_{X}(a, b). $$
\end{rk}

\begin{ass}\label{ass:correspondence}
For any nonempty subsets $A, B \subset X$ of a metric space $X$, we have
$$d^{H}_{X}(A, B) = \inf_{R_{0} \in \cR(A, B)} \sup_{(a, b) \in R_{0}} d_{X}(a, b).$$
\end{ass}

\begin{proof}
We carry out the proof assuming that both the Hausdorff distance and the infimum on the right-hand side of the equality to be proven are finite; the general case where one of the quantities (and consequently both) equals infinity is treated in an entirely similar manner.

The proof is based on the fact that one can explicitly construct a correspondence satisfying the required condition from the Hausdorff distance, and vice versa.

Let $d^{H}_{X}(A, B) = \e$. Then for any $\delta > 0$ and for each $a \in A$ there exists an element $b_a \in B$ such that $d_{X}(a, b_a) < \e + \delta$. Similarly, for each $b \in B$ there exists $a_b \in A$ such that $d_{X}(a_b, b) < \e + \delta$. Consider the correspondence $R_{0} \in \cR(A, B)$, defined by
$$ R_{0} = \bigl\{ (a, b_a) : a \in A \bigr\} \cup \bigl\{ (a_b, b) : b \in B \bigr\}. $$
By construction, $(a, b) \in R_{0}$ holds $d_{X}(a, b) < \e + \delta$, hence
$$\sup_{(a, b) \in R_{0}} d_{X}(a, b) \le \e + \delta.$$
Since $\delta > 0$ is arbitrary, we obtain
$$ \inf_{R_{0} \in \cR(A, B)} \sup_{(a, b) \in R_{0}} d_{X}(a, b) \le d^{H}_{X}(A, B). $$

Conversely, fix an arbitrary correspondence $R_{0} \in \cR(A, B)$ and let
$$\sup_{(a, b) \in R_{0}} d_{X}(a, b) = \e.$$
It follows from the definition of a correspondence that for any $a \in A$ there exists an element $b \in B$ such that $(a, b) \in R_{0}$. Then $\inf_{b' \in B} d_{X}(a, b') \le d_{X}(a, b) \le \e$. Taking the supremum over all $a \in A$, we obtain $\delta(A, B) \le \e$.

Symmetrically, for each $b \in B$ there exists $a \in A$ such that $(a, b) \in R_{0}$, which yields $\delta(B, A) \le \e$. Thus, $d^{H}_{X}(A, B) = \max\{\delta(A, B), \delta(B, A)\} \le \e$. Since the choice of $R_{0}$ is arbitrary, this implies
$$ d^{H}_{X}(A, B) \le \inf_{R_{0} \in \cR(A, B)} \sup_{(a, b) \in R_{0}} d_{X}(a, b). $$
\end{proof}

\begin{notation}
Let $(X, d_{X})$ be a metric space.
Denote by $\cP_{0}(X)$ the set of all its nonempty subsets.
We consider $\cP_{0}(X)$ as a space with a generalized pseudometric given by the Hausdorff distance $d_{X}^{H}$.
\end{notation}

\begin{notation}
By $\R_{+} = [0, \infty)$ we denote the non-negative real half-line with the standard metric.
\end{notation}

\section{Gromov--Hausdorff distance bounds: general case}

Let us estimate the Gromov--Hausdorff distance between wedge sums of pointed metric spaces.
Hereinafter, when we consider a product or a wedge sum of a family of metric spaces $\{X_{\a}\}_{\a \in A}$, we assume that this family $A$ contains at least two elements.

\begin{ass}[\cite{Me_GH_Product_2026}]
\begin{enumerate}
\item
Let arbitrary sequences of pointed metric spaces $\{(Y_n, y_{n})\}_{n \in \N}$ and correspondences
$$R_n \in \cR\bigl((X_n, x_{n}), (Y_n, y_{n})\bigr).$$
Let $1 \le p < \infty$, set
$$(X, x_{0}) = (l^{p}) \prod_{n = 1}^{\infty} (X_{n}, x_{n}), \
\quad (Y, y_{0}) = (l^{p}) \prod_{n = 1}^{\infty} (Y_{n}, y_{n}), \
$$
Consider the product of correspondences
$$R = \prod_{n = 1}^{\infty} R_{n} \in \cR\bigr((X, x_{0}), (Y,y_{0}) \bigr).$$
Then for $1\le p < \infty$ it holds that
$$\dis (R) \le \biggl( \sum_{n=1}^{\infty} \bigl( \dis(R_n) \bigr)^p \biggr)^{1/p}.$$
\item
Let arbitrary families of pointed metric spaces $\{ (X_{\a}, x_{\a}) \}_{\a \in A}$ and $\{(Y_{\a}, y_{\a})\}_{\a \in A}$ and correspondences
$R_{\a} \in \cR\bigl((X_{\a}, x_{\a}), (Y_{\a}, y_{\a})\bigr)$. Set
$$(X, x_0) = (l^{\infty}) \prod_{\a \in A} (X_{\a}, x_{\a}), \quad (Y, y_0) = (l^{\infty}) \prod_{\a \in A} (Y_{\a}, y_{\a}).$$
Consider the product of correspondences
$$R = \prod_{\a \in A} R_{\a} \in \cR\bigl((X, x_0), (Y, y_0)\bigr).$$
Then it holds that
$$\dis (R) \le \sup_{\a \in A} \dis(R_{\a}).$$
\end{enumerate}
\end{ass}

\begin{lem}
Let families of pointed metric spaces
$$\{ (X_{\a}, x_{\a}) \}_{\a \in A}$$
and $\{(Y_{\a}, y_{\a})\}_{\a \in A}$ and correspondences
$R_{\a} \in \cR((X_{\a}, x_{\a}), (Y_{\a},y_{\a}))$. Then for
$1\le p < \infty$ it holds that
$$\dis \biggl( (l^p)\bigvee_{\a \in A} R_{\a} \biggr) \le
\sup_{\substack{\a, \bt \in A \\ \a \ne \bt}} \biggl( \bigl(\dis(R_{\a})\bigr)^p + \bigl(\dis(R_{\bt})\bigr)^p \biggr)^{1/p},$$
and for $p=\infty$ it holds that
$$\dis \biggl( (l^{\infty})\bigvee_{\a \in A} R_{\a} \biggr) \le \sup_{\a \in A} \dis(R_{\a}).$$
\end{lem}

\begin{proof}
We prove this for $1 \le p < \infty$, the case $p = \infty$ is similar.
Set $(X, x_0) = (l^p) \prod_{\a \in A} (X_{\a}, x_{\a}), \ (Y, y_0) = (l^p) \prod_{\a \in A} (Y_{\a}, y_{\a})$.
Recall that $(x_{\a}, y_{\a}) \in R_{\a}$ for all $\a \in A$.
We will use the following observation: note that for any $\a\ne\bt$ there is an isometric embedding
$(l^p) R_{\a} \vee R_{\bt} \subset (l^p) R_{\a} \times R_{\bt}$, therefore
$\dis \bigl( (l^p) R_{\a} \vee R_{\bt} \bigr) \le \dis\bigl( (l^p) R_{\a} \times R_{\bt} \bigr)$.
Then
\begin{multline*}
\dis \biggl( (l^p) \bigvee_{\a \in A} R_{\a} \biggr) =
\sup_{\a, \bt \in A} \
\sup_{\substack{  (x_{\a}, y_{\a})\in R_{\a} \\ (\tilde{x}_{\bt}, \tilde{y}_{\bt})\in R_{\bt} }}
\bigl|d_{X}(x_{\a}, \tilde{x}_{\bt}) - d_{Y}(y_{\a}, \tilde{y}_{\bt}) \bigr| = \\
\
= \max \Biggl\{
\sup_{\substack{\a, \bt \in A \\ \a \ne \bt}} \
\sup_{\substack{  (x_{\a}, y_{\a})\in R_{\a} \\ (\tilde{x}_{\bt}, \tilde{y}_{\bt})\in R_{\bt} }}
\bigl|d_{X}(x_{\a}, \tilde{x}_{\bt}) - d_{Y}(y_{\a}, \tilde{y}_{\bt}) \bigr|, \
\sup_{\a \in A} \
\sup_{\substack{  (x_{\a}, y_{\a})\in R_{\a} \\ (\tilde{x}_{\a}, \tilde{y}_{\a})\in R_{\a} }}
\bigl|d_{X}(x_{\a}, \tilde{x}_{\a}) - d_{Y}(y_{\a}, \tilde{y}_{\a}) \bigr|
\Biggr\}  = \\
\
= \max \Biggl\{
\sup_{\substack{\a, \bt \in A \\ \a \ne \bt}} \dis((l^p) R_{\a} \vee R_{\bt}), \
\sup_{\a \in A} \dis(R_{\a})
\Biggr\}
\le \max \Biggl\{
\sup_{\substack{\a, \bt \in A \\ \a \ne \bt}} \dis((l^p) R_{\a} \times R_{\bt}), \
\sup_{\a \in A} \dis(R_{\a})
\Biggr\} \le \\
\
\le \max \Biggl\{
\sup_{\substack{\a, \bt \in A \\ \a \ne \bt}}
\biggl( \bigl(\dis(R_{\a})\bigr)^p + \bigl(\dis(R_{\bt})\bigr)^p \biggr)^{1/p}, \
\sup_{\a \in A} \dis(R_{\a})
\Biggr\}
= \sup_{\substack{\a, \bt \in A \\ \a \ne \bt}}
\biggl( \bigl(\dis(R_{\a})\bigr)^p + \bigl(\dis(R_{\bt})\bigr)^p \biggr)^{1/p}.
\end{multline*}
\end{proof}

\begin{thm}\label{thm:GH_upper_bound}
Let families of pointed metric spaces
$$\{ (X_{\a}, x_{\a}) \}_{\a \in A} \quad \text{ and } \quad \{(Y_{\a}, y_{\a})\}_{\a \in A}.$$
Then for
$1\le p < \infty$ it holds that
\begin{multline*}
d_{pGH} \biggl(
(l^p)\bigvee_{\a \in A} (X_{\a}, x_{\a}), \
(l^p)\bigvee_{\a \in A} (Y_{\a}, y_{\a}) \biggr) \le \\
\
\le \sup_{\substack{\a, \bt \in A \\ \a \ne \bt}} \biggl(
\Bigl( d_{pGH}\bigl((X_{\a}, x_{\a}), (Y_{\a}, y_{\a})\bigr) \Bigr)^p +
\Bigl( d_{pGH}\bigl((X_{\bt}, x_{\bt}), (Y_{\bt}, y_{\bt})\bigr) \Bigr)^p
\biggr)^{1/p},
\end{multline*}
and for $p=\infty$ it holds that
$$
d_{pGH} \biggl(
(l^{\infty})\bigvee_{\a \in A} (X_{\a}, x_{\a}), \
(l^{\infty})\bigvee_{\a \in A} (Y_{\a}, y_{\a}) \biggr) \le
\sup_{\a \in A}
d_{pGH}\bigl((X_{\a}, x_{\a}), (Y_{\a}, y_{\a})\bigr).
$$
\end{thm}

\begin{proof}
We prove this in the case $1 \le p < \infty$, the case $p = \infty$ is similar.
If for some $\a$ we have $d_{pGH}\bigl((X_{\a}, x_0), (Y_{\a}, y_0)\bigr) = \infty$, then there is nothing to prove. We will assume that
for all $\a \in A$ it holds that $d_{pGH}\bigl((X_{\a}, x_0), (Y_{\a}, y_0)\bigr) < \infty$.
Fix $\e > 0$ and consider $R_{\a} \in \cR\bigl((X_{\a},x_0), (Y_{\a}, y_0)\bigr)$ such that
$$\Bigl| 2d_{pGH}\bigl((X_{\a}, x_0), (Y_{\a}, y_0)\bigr) - \dis(R_{\a}) \Bigr| < \e.$$
In what follows, for brevity, we will write $X_{\a}$ instead of $(X_{\a}, x_0)$ and $Y_{\a}$ instead of $(Y_{\a}, y_0)$.
Then
\begin{multline*}
\biggl| \Bigl( \bigl(\dis(R_{\a})\bigr)^p + \bigl(\dis(R_{\bt})\bigr)^p \Bigr)^{1/p} -
\Bigl( \bigl(  2d_{pGH}(X_{\a}, Y_{\a}) \bigr)^p +
\bigl(  2d_{pGH}(X_{\bt}, Y_{\bt}) \bigr)^p \Bigr)^{1/p}\biggr| \le \\
\
\le \biggl(
\Bigl| \dis(R_{\a}) - 2d_{pGH}(X_{\a}, Y_{\a}) \Bigr|^p +
\Bigl| \dis(R_{\bt}) - 2d_{pGH}(X_{\bt}, Y_{\bt}) \Bigr|^p
\biggr)^{1/p} \le 2^{1/p}\e,
\end{multline*}
therefore
\begin{multline*}
2d_{pGH} \biggl(
(l^p)\bigvee_{\a \in A} (X_{\a}, x_0), \
(l^p)\bigvee_{\a \in A} (Y_{\a}, y_0) \biggr) \le
\dis \biggl( (l^p)\bigvee_{\a \in A} R_{\a} \biggr) \le \\
\
\le \sup_{\substack{\a, \bt \in A \\ \a \ne \bt}} \biggl( \bigl(\dis(R_{\a})\bigr)^p + \bigl(\dis(R_{\bt})\bigr)^p \biggr)^{1/p} \le \\
\
\le \sup_{\substack{\a, \bt \in A \\ \a \ne \bt}} \biggl( \bigl( 2d_{pGH}(X_{\a}, Y_{\a})\bigr)^p + \bigl(2d_{pGH}(X_{\bt}, Y_{\bt}))\bigr)^p \biggr)^{1/p} + 2^{1/p}\e.
\end{multline*}
Letting $\e \to 0$, we obtain the desired result.
\end{proof}

\begin{rk}
The bound in Theorem~\ref{thm:GH_upper_bound} makes sense (i.e., does not become infinite) if and only if
$$\sup_{\a \in A} d_{pGH}(X_{\a}, Y_{\a}) < \infty.$$
\end{rk}

\begin{dfn}
For an arbitrary pointed metric space $(X, d_{X}, x_0)$ we define its \emph{radius}
$$\Rad(X, x_0) = \sup_{x \in X}d_{X}(x, x_0).$$
\end{dfn}

\begin{rk}
It is clear from the definition that for any pointed metric space $(X, d_{X}, x_0)$ the following holds
$$\Rad(X, x_0) \le \diam(X) \le 2\Rad(X, x_0).$$
\end{rk}

\begin{thm}\label{thm:GH_lower_bound}
Let $\{(X_{\a}, x_{\a})\}_{\a \in A}$ be a family of pointed metric spaces. Let $(X,x_0) = (l^p)\vee_{\a \in A} (X_{\a}, x_{\a})$, $(Y, y_0) = (l^p)\vee_{\a \in A} (Y_{\a}, y_{\a})$.

Then for $1 \le p < \infty$ it holds that
\begin{multline*}
d_{pGH}\bigl((X,x_0), (Y, y_0)\bigr) \ge
\sup_{\a \in A} \Biggl( d_{pGH}\bigl((X_{\a}, x_{\a}), (Y_{\a}, y_{\a})\bigr) - \\
\
- \frac{1}{2}
\max\Biggl\{
\sup_{S \subset A \setminus \{ \a \} \: \#S \le 2}
\biggl( \sum_{\bt \in S} \bigl( \Rad(X_{\bt}, x_{\bt}) \bigr)^p \biggr)^{1/p}, \
\sup_{\bt \ne \alpha} \diam(X_{\bt})
\Biggr\}
- \\
- \frac{1}{2} \max\Biggl\{ \sup_{S \subset A \setminus \{ \a \} \: \#S \le 2}
\biggl( \sum_{\bt \in S} \bigl( \Rad(Y_{\bt}, y_{\bt}) \bigr)^p \biggr)^{1/p}, \
\sup_{\bt \ne \alpha} \diam(Y_{\bt}) \Biggr\}
\Biggr).
\end{multline*}
For $p = \infty$ it holds that
\begin{multline*}
d_{pGH}\bigl((X,x_0), (Y, y_0)\bigr) \ge \\
\ge \sup_{\a \in A} \Biggl( d_{pGH}\bigl((X_{\a}, x_{\a}), (Y_{\a}, y_{\a})\bigr)
- \frac{1}{2} \sup_{\bt \ne \a} \
\diam(X_{\bt}) -
\frac{1}{2} \sup_{\bt \ne \a} \
\diam(Y_{\bt})
\Biggr).
\end{multline*}
\end{thm}

\begin{proof}
Consider the case $1 \le p < \infty$, the case $p = \infty$ is similar.
Define a map $p_{\a} \: X \to X_{\a}$ given as follows:
if $x \in X_{\a}$, then $p_{\a}(x) = x$, and if $x \in X_{\bt}, \ \bt \ne \a$, then $p_{\a}(x) = x_0$. The map $p_{\a} \: Y \to Y_{\a}$ is defined similarly.
For a given correspondence $R \in \cR((X, x_0), (Y, y_0))$ we define
$$p_{\a}R = \Bigl\{ (x', y') \in X_{\a} \times Y_{\a} \ \Big| \ \exists (x, y) \in X \times Y \: \bigl(p_{\a}(x), p_{\a}(y)\bigr) = (x', y') \Bigr\},$$
then $p_{\a}R$ is a correspondence
$$p_{\a}R \in \cR\bigl((X_{\a}, x_{\a}), (Y_{\a}, y_{\a})\bigr).$$
For each $\a \in A$ we obtain
\begin{multline*}
d_{pGH}\bigl((X_{\a}, x_0), (Y_{\a}, y_0)\bigr) = \frac{1}{2} \inf_{R \in \cR((X_{\a}, x_0), (Y_{\a}, y_0))} \dis(R)
\le \frac{1}{2} \inf_{R \in \cR\left((X, x_0),(Y, y_0)\right)} \dis(p_{\a}R) = \\
\
= \frac{1}{2} \inf_{R \in \cR((X, x_0),(Y, y_0))}
\sup_{\substack{(x, y) \in R \\ (\tilde{x}, \tilde{y}) \in R}}
\bigl| d_{X_{\a}}(p_{\a}x, p_{\a}\tilde{x}) - d_{Y_{\a}}(p_{\a}y, p_{\a}\tilde{y})\bigr| \le \\
\
\le \frac{1}{2} \inf_{R \in \cR((X, x_0),(Y, y_0))}
\sup_{\substack{(x, y) \in R \\ (\tilde{x}, \tilde{y}) \in R}}
\Bigl(
\bigl| d_{X}(x, \tilde{x}) - d_{Y}(y, \tilde{y}) \bigr| + \\ +
\bigl| d_{X}(x, \tilde{x}) - d_{X_{\a}}(p_{\a}x, p_{\a}\tilde{x}) \bigr| +
\bigl| d_{Y}(y, \tilde{y}) - d_{Y_{\a}}(p_{\a}y, p_{\a}\tilde{y}) \bigr|
\Bigr) \le \\
\
\le d_{pGH}((X, x_0), (Y, y_0)) +
\frac{1}{2} \inf_{R \in \cR((X, x_0), (Y, y_0))}
\sup_{x, \tilde{x} \in X}
\bigl| d_{X}(x, \tilde{x}) - d_{X_{\a}}(p_{\a}x, p_{\a}\tilde{x}) \bigr| + \\
\
+ \frac{1}{2} \inf_{R \in \cR((X, x_0), (Y, y_0))}
\sup_{y, \tilde{y} \in Y}
\bigl| d_{Y}(y, \tilde{y}) - d_{Y_{\a}}(p_{\a}y, p_{\a}\tilde{y}) \bigr|.
\end{multline*}
Now consider the expression
$$G = \bigl| d_{X}(x, \tilde{x}) - d_{X_{\a}}(p_{\a}x, p_{\a}\tilde{x}) \bigr|$$
and estimate it in each of the following cases.
\begin{itemize}
\item
Let $x \in X_{\a}$, $\tilde{x} \in X_{\a}$. Then $p_{\a}x = x$, $p_{\a}\tilde{x} = \tilde{x}$, therefore $G = 0$.
\item
Let $x \in X_{\a}$, $\tilde{x} \in X_{\bt}$, $\bt \ne \a$.
Then
\begin{multline*}
G = \bigl| d_{X}(x, \tilde{x}) - d_{X_{\a}}(p_{\a}x, p_{\a}\tilde{x}) \bigr| =
\biggl|
\Bigl\| \bigl(d_{X_{\a}}(x, x_{\a}), d_{X_{\bt}}(\tilde{x}, x_{\bt}) \bigr) \Bigr\|_{l^p} -
\Bigl\| \bigl(d_{X_{\a}}(x, x_{\a}), 0 \bigr) \Bigr\|_{l^p}
\biggr| \le \\
\
\le \Bigl\| \bigl(0, d_{X_{\bt}}(\tilde{x}, x_{\bt})\bigr) \Bigr\|_{l^p} = d_{X_{\bt}}(\tilde{x}, x_{\bt}) \le \Rad(X_{\bt}, x_{\bt}).
\end{multline*}
\item
Let $x \in X_{\bt}$, $\tilde{x} \in X_{\g}$, $\bt \ne \a$, $\g \ne \a$, $\g \ne \bt$.
Then
\begin{multline*}
G = \bigl| d_{X}(x, \tilde{x}) - d_{X_{\a}}(p_{\a}x, p_{\a}\tilde{x})\bigr| = d_{X}(x, \tilde{x}) =
\Bigl( \bigl(d_{X_{\bt}}(x, x_{\bt})\bigr)^p + \bigl(d_{X_{\g}}(\tilde{x}, x_{\g})\bigr)^p \Bigr)^{1/p} \le \\
\
\le \Bigl( \bigl(\Rad(X_{\bt}, x_{\bt})\bigr)^p + \bigl(\Rad(X_{\g}, x_{\g})\bigr)^p \Bigr)^{1/p}.
\end{multline*}
\item
Let $x, \tilde{x} \in X_{\beta}, \ \beta \ne \alpha$. Then
$$G = \bigl|d_{X}(x, \tilde{x}) - d_{X_{\a}}(p_{\a}x, p_{\a}\tilde{x}) \bigr| = d_{X_{\bt}}(x, \tilde{x}) \le \diam(X_{\bt}).$$
\end{itemize}
Therefore, the following inequality holds
\begin{multline*}
\inf_{R \in \cR((X, x_0), (Y, y_0))}
\sup_{\substack{(x, y) \in R \\ (\tilde{x}, \tilde{y}) \in R}}
\bigl| d_{X}(x, \tilde{x}) - d_{X_{\a}}(p_{\a}x, p_{\a}\tilde{x}) \bigr| \le
\\
\le \max \Biggl\{
\sup_{S \subset A \setminus \{ \a \} \: \#S \le 2}
\biggl( \sum_{\vk \in S} \bigl( \Rad(X_{\vk}, x_{\vk}) \bigr)^{p} \biggr)^{1/p}, \
\sup_{\bt \ne \a} \diam(X_{\bt})
\Biggr\}
\end{multline*}
The latter expression is estimated similarly:
$\bigl| d_{Y}(y, \tilde{y}) - d_{Y_{\a}}(p_{\a}y, p_{\a}\tilde{y}) \bigr|$.

Thus, for each $\a \in A$ it holds that
\begin{multline*}
d_{pGH}\bigl((X_{\a}, x_{\a}), (Y_{\a}, y_{\a})\bigr) \le d_{pGH}\bigl((X, x_0), (Y, y_0)\bigr) + \\
\
+ \frac{1}{2} \max \Biggl \{
\sup_{S \subset A \setminus \{ \a \} \: \#S \le 2} \biggl( \sum_{\vk \in S} \bigl( \Rad(X_{\vk}, x_{\vk}) \bigr)^p \biggr)^{1/p}, \
\sup_{\bt \ne \a} \diam(X_{\bt})
\Biggr\}
+ \\ + \frac{1}{2} \max \Biggl \{
\sup_{S \subset A \setminus \{ \a \} \: \#S \le 2} \biggl( \sum_{\vk \in S} \bigl( \Rad(Y_{\vk}, y_{\vk}) \bigr)^p \biggr)^{1/p}, \
\sup_{\bt \ne \a} \diam(Y_{\bt})
\Biggr\},
\end{multline*}
from which the required inequality immediately follows by passing to the supremum over $\a$.
\end{proof}

\begin{rk}
The condition $\# S \le 2$ over which the supremum is taken in Theorem~\ref{thm:GH_lower_bound}, is due to the fact that the projection $p_{\a}$ `collapses' all components except $X_{\a}$ to the base point, so when estimating the distortion for points $x \in X_{\bt}$ and $\tilde{x} \in X_{\g}$ it accumulates from at most two branches of the wedge sum.
\end{rk}

\section{Bound for $l^p$-wedge sums on a common family of base spaces}

\begin{notation}
For any pointed metric space $(X, d_{X}, x_{0})$ and number $\lambda > 0$ we denote by $\lambda (X, d_{X}, x_{0})$ or $\lambda (X, x_0)$ the pointed metric space $(X, \lambda d_{X}, x_{0})$.
\end{notation}

Recall the following proposition.

\begin{prop}[\cite{Me_GH_Product_2026}] \label{prop:matan}
\begin{enumerate}
    \item Let $a=\{a_n\}_{n \in \N}\in l^p$, $b=\{b_n\}_{n \in \N}\in l^p$ be two sequences, $1 \le p < \infty$.
    Define the function $\xi_p \: [0, 1]^{\N} \to \R$ given as follows\/\rom:
    $$\xi_p (\vk)=\Biggl| \biggl(\sum_{n=1}^{\infty} |a_n \vk_n|^p \biggl)^{1/p} -
    \biggl(\sum_{n=1}^{\infty} |b_n \vk_n|^p \biggl)^{1/p}
    \Biggr|. $$
    Then $$\sup_{\vk \in [0, 1]^{\N}} \xi_p (\vk)=\sup_{\vk \in \{0, 1\}^{\N}} \xi_p (\vk)=\sup_{S \subset \N} \ \Biggl|
    \biggl(\sum_{n \in S} |a_n|^p \biggr)^{1/p}- \biggl(\sum_{n \in S} |b_n|^p \biggr)^{1/p} \Biggr|.$$

    \item Let $a = \{a_{\alpha}\}_{\alpha \in A}$, $b = \{b_{\alpha}\}_{\alpha \in A}$ be families of numbers such that $\sup_{\alpha \in A} a_{\alpha} < \infty$ and $\sup_{\alpha \in A} b_{\alpha} < \infty$.
    Define the function $\xi_{\infty} :\ [0, 1]^A \to \R$ given as follows\/\rom:
    $$\xi_{\infty}(\vk) = \bigg| \sup_{\alpha \in A}|a_{\alpha}\vk_{\alpha}| - \sup_{\alpha \in A}|b_{\alpha}\vk_{\alpha}|\bigg|.$$
    Then
    $$\sup_{\vk \in [0, 1]^A} \xi_{\infty}(\vk)=\sup_{\vk \in \{0, 1\}^A} \xi_{\infty}(\vk)= \sup_{S \subset A}\bigg| \sup_{\alpha \in S}|a_{\alpha}| - \sup_{\alpha \in S}|b_{\alpha}|\bigg|.$$
\end{enumerate}
\end{prop}

Using this proposition, we prove the following theorem.

\begin{thm}\label{thm:linear}
Let $\{(W_{\a}, w_{\a})\}_{\a \in A}$ be an arbitrary family of pointed metric spaces of cardinality $\# A \ge 2$, such that $\Rad(W_{\a}, w_{\a}) = 1$ for all $\a \in A$.
Let $$a = \{a_{\a}\}_{\a \in A}, a_{\a} \ge 0, \quad b = \{b_{\a}\}_{\a \in A}, b_{\a} \ge 0.$$
Let $1 \le p \le \infty$.
Set
$$(X, x_0) = (l^p) \bigvee_{\bt \in A}a_{\bt}(W_{\bt}, w_{\bt}), \
(Y, y_0) = (l^p) \bigvee_{\bt \in A}b_{\bt}(W_{\bt}, w_{\bt}).$$
Then for $1 \le p < \infty$ it holds that
$$2d_{pGH}\bigl((X, x_0), (Y, y_0)\bigr) \le
\sup_{\a \ne \bt} \max \biggl\{
|a_{\a} - b_{\a}| \ \diam(W_{\a}), \
\Bigl| (a_{\a}^p+a_{\bt}^p)^{1/p} - (b_{\a}^p+b_{\bt}^p)^{1/p} \Bigr|
\biggr\}, $$
and in the case $p = \infty$ it holds that
$$2d_{pGH}\bigl((X, x_0), (Y, y_0)\bigr) \le
\sup_{\a \ne \bt} \max \biggl\{
|a_{\a} - b_{\a}| \ \diam(W_{\a}), \
\Bigl| \max\{a_{\a}, a_{\bt}\} - \max\{b_{\a}, b_{\bt} \} \Bigr|
\biggr\}.$$
\end{thm}

\begin{proof}
Let $1 \le p < \infty$ (the case $p = \infty$ is treated similarly).
Consider the identity map $g \: X \to Y, \ g(x) = x$.
Then
\begin{multline*}
2d_{pGH}\bigl((X, x_0), (Y, y_0)\bigr) \le \dis(g) = \\
\
= \sup_{w, \tilde{w} \in X}
\bigl| d_{X}(w, \tilde{w}) - d_{Y}(w, \tilde{w})\bigr|
= \sup_{\a, \bt} \sup_{\substack{w \in W_{\a} \\ \tilde{w} \in W_{\bt}}}
\bigl| d_{X}(w, \tilde{w}) - d_{Y}(w, \tilde{w})\bigr| = \\
\
= \max \biggl\{
\sup_{\a \ne \bt}
\sup_{\substack{w \in W_{\a} \\ \tilde{w} \in W_{\bt}}}
\bigl| d_{X}(w, \tilde{w}) - d_{Y}(w, \tilde{w})\bigr|, \
\sup_{\a \in A} \sup_{w, \tilde{w} \in W_{\a}}
\bigl| d_{X}(w, \tilde{w}) - d_{Y}(w, \tilde{w})\bigr|
\biggr\} = \\
\
= \max \Biggl\{
\sup_{\a \ne \bt}
\sup_{\substack{w \in W_{\a} \\ \tilde{w} \in W_{\bt}}}
\biggl| \Bigl(
\bigl(a_{\a}d_{W_{\a}}(w, w_{\a}) \bigr)^p +
\bigl(a_{\bt}d_{W_{\bt}}(\tilde{w}, w_{\bt}) \bigr)^p
\Bigr)^{1/p} - \\
- \Bigl( \bigl(b_{\a}d_{W_{\a}}(w, w_{\a}) \bigr)^p +
\bigl(b_{\bt}d_{W_{\bt}}(\tilde{w}, w_{\bt}) \bigr)^p
\Bigr)^{1/p}
\biggr|, \
\sup_{\a \in A} |a_{\a} - b_{\a}| \ \diam(W_{\a})
\Biggr\}
\end{multline*}

Since $\Rad(W_{\a}, w_{\a}) = 1$, the range of the distance function $\vk_{\a} = d_{W_{\a}}(w, w_{\a})$ lies in the interval $[0, 1]$.
Consequently, the supremum over $w \in W_{\a}, \ \tilde{w} \in W_{\bt}$ is majorized by the supremum over the variables $\vk_{\a}, \vk_{\bt} \in [0, 1]$. Thus, continuing our chain of inequalities and applying Proposition~\ref{prop:matan}, we obtain:

\begin{multline*}
2d_{pGH}\bigl((X, x_0), (Y, y_0)\bigr)
\le \max \Biggl\{
\sup_{\a \ne \bt} \
\sup_{0 \le \vk_{\a}, \vk_{\bt} \le 1} \
\biggl| \Bigl(
\bigl(a_{\a}\vk_{\a} \bigr)^p +
\bigl(a_{\bt}\vk_{\bt} \bigr)^p
\Bigr)^{1/p} - \\
- \Bigl( \bigl(b_{\a}\vk_{\a} \bigr)^p +
\bigl(b_{\bt}\vk_{\bt} \bigr)^p
\Bigr)^{1/p}
\biggr|, \
\sup_{\a \in A} |a_{\a} - b_{\a}| \ \diam(W_{\a})
\Biggr\} = \\
\
= \max \Biggl\{
\sup_{\a \ne \bt} \
\max \biggl\{  |a_{\a} - b_{\a}|, |a_{\bt} - b_{\bt}|,
\Bigl| (a_{\a}^p+a_{\bt}^p)^{1/p} - (b_{\a}^p+b_{\bt}^p)^{1/p} \Bigr|  \biggr\}, \
\sup_{\a \in A} |a_{\a} - b_{\a}| \ \diam(W_{\a})
\Biggr\} = \\
\
= \sup_{\a \ne \bt} \max \biggl\{
|a_{\a} - b_{\a}| \ \diam(W_{\a}), \
\Bigl| (a_{\a}^p+a_{\bt}^p)^{1/p} - (b_{\a}^p+b_{\bt}^p)^{1/p} \Bigr|
\biggr\}.
\end{multline*}
\end{proof}

\section{Isometric Embedding of $l^{\infty} \hookrightarrow \cGH$}

In this section, we consider the Banach space
$$l^{\infty} = \biggl\{ (x_{1}, x_{2}, x_{3}, \dots ) \ \Big| \ \sup_{n \in \N} |x_{n}| < \infty \biggr\}, \quad \|x\|_{l^{\infty}} = \sup_{n \in \N} |x_{n}|,$$
and prove the existence of an isometric embedding $l^{\infty} \hookrightarrow \cGH$.

Notice that for any pointed metric spaces $(X, x_{0}), (Y, y_{0})$ we have
$$d_{pGH}\bigl((X, x_0), (Y, y_{0})\bigr) \ge d_{GH}(X, Y),$$
where the right-hand side is the standard Gromov--Hausdorff distance that does not take the base point into account.
Thus, Theorem~\ref{thm:GH_upper_bound} established in Section~2 also bounds the standard Gromov--Hausdorff distance whenever the pointed Gromov--Hausdorff distances between the components of the wedge sum are known.

Recall that $\cP_{0}(\R_{+})$ denotes the space of nonempty subsets of the nonnegative half-line endowed with the Hausdorff distance.
Consider the subspace $\cP_{00}(\R_{+})$ consisting of sets containing the point $0$:
$$\cP_{00}(\R_{+}) = \bigl\{K\in \cP_{0}(\R_{+}) \mid 0\in K \bigr\}.$$
We now proceed to construct an isometric embedding of the space $\cP_{00}(\R_{+})$ into the class $\cGH$.

\begin{notation}
Let $\aleph$ be an arbitrary cardinal number.
Let $\Delta_{\aleph}$ be a metric space of cardinality $\aleph$ in which all nonzero distances are equal to $1$.
\end{notation}

\begin{rk}\label{rk:DeltaIso}
Let $\aleph$ be an arbitrary infinite cardinal.
Let $I$ be a set of cardinality $\aleph$.
Then the following spaces are isometric:
$$(l^{\infty}) \bigvee_{i \in I} \Delta_{\aleph} = \Delta_{\aleph}.$$
\end{rk}

\begin{proof}
Indeed, by definition of the $l^{\infty}$-wedge sum, the distance between any two points distinct from the base point equals $\max(1, 1) = 1$.
Thus, the wedge sum is a metric simplex of cardinality $\aleph \times \aleph = \aleph$.
\end{proof}

The following statement is well known.
\begin{ass}[\cite{BurBurIva}]
Let $X, Y$ be bounded metric spaces and let $\lambda > 0$.
Then
$$d_{GH}(\lambda X, \lambda Y) = \lambda d_{GH}(X, Y).$$
\end{ass}

Here and throughout this section, we will consider the classical Gromov--Hausdorff distance $d_{GH}$ between metric spaces (instead of the pointed distance $d_{pGH}$).
Even when discussing the distance between $l^p$-wedge sums, we will completely ignore the base point and treat the space as a standard metric space.

\begin{lem}
Let $\aleph \ge \mathfrak{c} = 2^{\aleph_0}$ be an arbitrary cardinal not less than the continuum.
To each nonempty set $C \subset \R_{+}$, we associate the $l^{\infty}$-wedge sum of metric simplices $\Delta_\aleph$\rom:
$$X_C = (l^{\infty}) \bigvee_{c \in C} c\Delta_\aleph.$$
Let $A, B \subset \R_{+}$ be arbitrary nonempty subsets.
Set $A_0 = \{0\} \cup A$ and $B_0 = \{0\} \cup B$.
Then the Gromov--Hausdorff distance between the spaces $X_A$ and $X_B$ is exactly equal to\/\rom:
$$d_{GH}(X_A, X_B) = \frac{1}{2} d_{\R_{+}}^{H}(A_0, B_0).$$
\end{lem}

\begin{proof}
For brevity, denote $X = X_A$ and $Y = X_B$.
Recall the definition of the Gromov--Hausdorff distance:
$$d_{GH}(X, Y) = \frac{1}{2} \inf_{\substack{f: X \to Y \\ g: Y \to X}} \max \bigl\{ \dis(f), \ \dis(g), \  \codis(f, g) \bigr\}.$$
We divide the proof of the assertion into two steps.

\paragraph{Step 1.} We show that $d_{GH}(X, Y) \ge \frac{1}{2} d^{H}_{\R_{+}}(A_0, B_0)$.

Let $d^{H}_{\R_{+}}(A_0, B_0) < \infty$ (the case where this is not so is trivial). By definition of the Hausdorff distance:
$$d^{H}_{\R_{+}}(A_0, B_0) = \max \left\{ \sup_{a \in A_0} \inf_{b \in B_0} |a - b|, \, \sup_{b \in B_0} \inf_{a \in A_0} |a - b| \right\}.$$
Fix an arbitrary $\e > 0$. Without loss of generality, assume that there exists a number $a_0 \in A_0$ such that for all $b \in B_0$ the inequality $|a_0 - b| \ge d^H_{\R_{+}}(A_0, B_0) - \e$ holds (if not, swap the sets $A$ and $B$ and proceed with the proof).

Let us verify that the sets $A_0$ and $B_0$ contain all values of distances between any pairs of points in $X$ and $Y$, respectively:
$$A_{0} = \bigl\{d_{X}(x_1, x_2) \mid x_1, x_2 \in X\bigr\}, \quad
B_{0} = \bigl\{d_{Y}(y_1, y_2) \mid y_1, y_2 \in Y\bigr\}.$$
We justify the equality for $A_0$ (the argument for $B_0$ is similar).
Let $x, \tilde{x} \in X_A = (l^\infty) \vee_{c \in A} c\Delta_\aleph$. By definition of the wedge sum metric:
\begin{itemize}
\item
if $x, \tilde{x}$ belong to the same component $c\Delta_\aleph$, then $d_X(x, \tilde{x}) = c \cdot d_{\Delta_\aleph}(x, \tilde{x})$. Since distances in a metric simplex are only $0$ or $1$, we have $d_X(x, \tilde{x}) \in \{0, c\} \subset A_0$;
\item
if $x \in c_1\Delta_\aleph$ and $\tilde{x} \in c_2\Delta_\aleph$ with $c_1 \neq c_2$, then $d_X(x, \tilde{x}) = \max\{d_{c_1\Delta_\aleph}(x, x_0), d_{c_2\Delta_\aleph}(\tilde{x}, x_0)\} = \max\{c_1, c_2\}$. Since $c_1, c_2 \in A$, their maximum also belongs to $A \subset A_0$.
\end{itemize}
Thus, all pairwise distances in $X$ constitute precisely the set $A_0$.

Now consider an arbitrary map $f: X \to Y$.
Let points $x_1, x_2 \in X$ be such that $d_{X}(x_1, x_2) = a_0$.
Set $b^{*} := d_{Y}(f(x_1), f(x_2)) \in B_0$.
Consequently,
$$\dis(f) \ge |a_0 - b^*| \ge \inf_{b \in B_0} |a_0 - b| \ge d^{H}_{\R_{+}}(A_0, B_0) - \e.$$
Since $\e > 0$ is arbitrary, we obtain the lower bound:
$$d_{GH}(X, Y) \ge \frac{1}{2} \dis(f) \ge \frac{1}{2} d^{H}_{\R_{+}}(A_0, B_0).$$

\paragraph{Step 2.} We show that, conversely, $d_{GH}(X, Y) \le \frac{1}{2} d^{H}_{\R_{+}}(A_0, B_0)$.

The geometric idea behind the proof of the upper bound is to `expand' the original wedge sums according to Remark~\ref{rk:DeltaIso}.
Instead of directly comparing the spaces $X$ and $Y$, we construct auxiliary $l^{\infty}$-wedge sums isometric to the original ones, but indexed by pairs $(a, b)$ from the correspondence $R_0$. In such a space, the components with indices $a$ and $b$ turn out to be rigidly linked to each other, allowing one to easily construct the desired correspondence.

Now, fix $\e > 0$ and define the set $R_0 \subset A_0 \times B_0$ as follows:
$$R_0 = \bigl\{ (a, b) \in A_0 \times B_0 : |a - b| \le d^{H}_{\R_{+}}(A_0, B_0) + \e \bigr\}.$$
Proposition~\ref{ass:correspondence} implies that the projections of $R_0$ onto $A_0$ and $B_0$ are surjective, that is, $R_0 \in \cR(A_0, B_0)$.
Now we can rewrite (up to isometry) the original spaces via this correspondence as follows:
$$X = (l^{\infty}) \bigvee_{(a,b) \in R_0} a\Delta_\aleph, \quad Y = (l^{\infty}) \bigvee_{(a,b) \in R_0} b\Delta_\aleph.$$
Then by Theorem~\ref{thm:GH_upper_bound}, we obtain:
$$d_{GH}(X, Y) \le d_{pGH}(X, Y) \le \sup_{(a,b) \in R_0} d_{pGH}(a\Delta_\aleph, b\Delta_\aleph) = \sup_{(a,b) \in R_0} \frac{|a - b|}{2} \le \frac{1}{2} \bigl( d^{H}_{\R_{+}}(A_0, B_0) + \e \bigr),$$
where the equality $d_{pGH}(a \Delta_{\aleph}, b \Delta_{\aleph}) = d_{GH}(a \Delta_{\aleph}, b \Delta_{\aleph}) = |a-b|/2$ holds because the base point of the simplex is equivalent to any other vertex.
Letting $\e \to 0^+$, we obtain the required upper bound $d_{GH}(X, Y) \le \frac{1}{2} d^{H}_{\R_{+}}(A_0, B_0)$.

Combining Steps~1 and~2, we arrive at the equality.
\end{proof}

Thus, the following theorem is proved.

\begin{thm}
Let $\aleph$ be a cardinal not less than the continuum.
Then there exists an isometric embedding\/\rom:
$$\cP_{00}(\R_{+}) \hookrightarrow \cGH, \quad A \mapsto (l^{\infty}) \bigvee_{a \in A} 2a \Delta_{\aleph}.$$
\end{thm}

\begin{notation}
For any number $a \in \R_{+}$, set $a_{+} := \max\{a, 0\}$.
\end{notation}

\begin{lem}
There exists an isometric embedding
$l^{\infty} \hookrightarrow \cP_{00}(\mathbb{R}_{+}).$
\end{lem}

\begin{proof}
Let $x = \{x_n\}_{n=1}^{\infty} \in l^{\infty}$ be an arbitrary bounded sequence of real numbers. By definition of the space $l^\infty$, there exists a constant $M > 0$ such that $|x_n| \le M$ for all $n \in \N$.
Fix a bijection (renumbering) $k: \N \times \N \to \N$ assigning a natural number $k(n, m)$ to each pair of indices $(n, m)$. For each $k \in \mathbb{N}$, consider the family of intervals of the form:
$$ I_k = \bigl[2^k - k, \, 2^k + k\bigr]. $$
To each sequence $x \in l^\infty$, we associate the set $K_x \subset \mathbb{R}_+$ as follows:
$$ K_{x} = \{0\} \cup \biggl\{ \Bigl(2^{k(n,m)} + x_n\Bigr)_{+} \ \Big| \ n, m \in \N \biggr\}. $$

The rationale behind this construction of the set $K_x$ is that for sufficiently large $k$, the quantity $2^k + x_n$ is strictly positive; therefore, the operation $(\cdot)_+$ ensures that $K_x \subset \R_+$.
The intervals $I_k$ are chosen such that for large $k$ they are mutually disjoint for distinct $k$, which allows one to unambiguously `isolate' the components of the sequence $x$ within the set $K_x$ from each other.

The map $x \mapsto K_x$ is well-defined as a map $l^{\infty} \to \cP_{00}(\R_{+})$.
We now demonstrate that the map $x \mapsto K_x$ is an isometric embedding.

\paragraph{Step 1.} We prove the upper bound $d^{H}_{\R_{+}}(K_x, K_y) \le \|x - y\|_{l^\infty}$.

Set
$$R_{0} = \bigl\{(0, 0)\bigr\} \cup \biggl\{ \Bigl( \bigl( 2^{k(n,m)} + x_{n}\bigr)_{+}, \ \bigl( 2^{k(n,m)} + y_{n}\bigr)_{+} \Bigr) \ \Big| \ n, m \in \N \ \biggr\}.$$
By definition of the sets $K_x$ and $K_y$, we have $R_0 \in \cR(K_x, K_y)$.
Using the elementary property of real numbers $|\alpha_{+} - \beta_{+}| \le |\alpha - \beta|$, that is, the $1$-Lipschitz property of the map $\alpha \mapsto \alpha_+$, we obtain:
$$\sup_{(a, b) \in R_{0}} |a - b| = \sup_{n, m \in \N} \ \Bigl|\bigl(2^{k(n, m)} + x_{n}\bigr)_{+} - \bigl(2^{k(n, m)} + y_{n}\bigr)_{+}\Bigr| \le \sup_{n \in \N} |x_{n} - y_{n}| = \|x - y\|_{l^{\infty}}.$$
Consequently, by Proposition~\ref{ass:correspondence},
$$d^{H}_{\R_{+}}(K_x, K_y) \le \sup_{(a, b) \in R_{0}} |a - b| \le \|x - y\|_{l^{\infty}}.$$

\paragraph{Step 2.} We prove the lower bound $d^{H}_{\R_{+}}(K_x, K_y) \ge \|x - y\|_{l^\infty}$.

If necessary, we increase the constant $M > 0$ so that
$$M \ge \max \bigl\{ \|x\|_{l^{\infty}}, \|y\|_{l^{\infty}} \bigr\}$$ holds.
Fix $\e > 0$.
By definition of the supremum for the norm in $l^\infty$, there exists an index $n_0 \in \N$ such that:
$$ |x_{n_0} - y_{n_0}| > \|x - y\|_{l^\infty} - \e. $$
Choose a number $m_0 \in \N$ such that
$$2^{k(n_0, m_0)} \ge 100 \max \bigl\{ \|x-y\|_{l^{\infty}}, \ M \bigr\}.$$

Set $a^{*} = 2^{k(n_0,m_0)} + x_{n_0} \in K_{x}$. Notice that the interval
$$ J = \bigl[ a^{*} - 2\|x-y\|_{l^{\infty}}, \ a^{*} + 2\|x-y\|_{l^{\infty}} \bigr] $$
contains the point $b^{*} = 2^{k(n_0,m_0)} + y_{n_0} \in K_y$ since
$$ |a^{*} - b^{*}| = |x_{n_0} - y_{n_0}| \le \|x-y\|_{l^{\infty}}. $$

Let us show that no other points of the set $K_y$ belong to this interval, that is, the following equality holds:
\begin{equation}\label{eq:intersection_single}
\bigl[ a^{*} - 2\|x-y\|_{l^{\infty}}, \ a^{*} + 2\|x-y\|_{l^{\infty}} \bigr] \cap K_{y} = \bigl\{ 2^{k(n_0, m_0)} + y_{n_0} \bigr\}.
\end{equation}
Indeed, consider an arbitrary element $b = 2^{k(n, m)} + y_n \in K_y \setminus \{b^{*}\}$, where $(n, m) \neq (n_0, m_0)$. For brevity, denote $k_0 = k(n_0, m_0)$ and $k = k(n, m)$. Two cases are possible.
\begin{itemize}
\item
If $k < k_0$, then by the choice of the number $m_0$ we have:
$$ b \le 2^{k_0-1} + M < 2^{k_0} - 2\|x-y\|_{l^{\infty}} - M \le a^{*} - 2\|x-y\|_{l^{\infty}}, $$
where the upper bound on $b$ and the lower bound on the left endpoint of the interval $J$ are guaranteed by the condition $2^{k_0} \ge 100 \max\{\|x-y\|_{l^\infty}, M\}$. Consequently, $b < \inf J$.
\item
If $k > k_0$, we similarly obtain:
$$ b \ge 2^{k_0+1} - M > 2^{k_0} + 2\|x-y\|_{l^{\infty}} + M \ge a^{*} + 2\|x-y\|_{l^{\infty}}, $$
which implies that $b > \sup J$.
\end{itemize}

Thus, for any $b \in K_y \setminus \{b^{*}\}$ we have $|a^{*} - b| > 2\|x-y\|_{l^{\infty}}$, which proves equality~\eqref{eq:intersection_single}.

From the resulting equality it directly follows that
$$ d_{\mathbb{R}_+}^{H}(K_x, K_y) \ge \delta(K_x, K_y) \ge \inf_{b \in K_y} |a^{*} - b| = |a^{*} - b^{*}| = |x_{n_0} - y_{n_0}| > \|x-y\|_{l^{\infty}} - \e. $$
Taking the limit as $\e \to 0$, we obtain the required lower bound $d_{\mathbb{R}_+}^{H}(K_x, K_y) \ge \|x-y\|_{l^{\infty}}$.

Combining the results of both steps, we arrive at the equality
$$d^{H}_{\R_{+}}(K_x, K_y) = \|x - y\|_{l^\infty}.$$
\end{proof}

Thus, the following theorem is proved.

\begin{thm}
There exists an isometric embedding
$l^{\infty} \hookrightarrow \cGH$.
\end{thm}

\begin{cor}
Every separable metric space $X$ isometrically embeds into $\cGH$.
\end{cor}

\begin{proof}
There exist isometric embeddings $X \hookrightarrow l^{\infty}$ (the Kuratowski embedding) and $l^{\infty} \hookrightarrow \cGH$.

The Kuratowski embedding is constructed as follows.
Let $\{z_{n}\}_{n \in \N} \subset X$ be a countable dense subset.
Choose a base point $x_0 \in X$.
Then the map
$x \mapsto \bigl(d_{X}(x, z_n) - d_{X}(x_0, z_n)\bigr)_{n=1}^{\infty}$ defines an isometric embedding $X \hookrightarrow l^{\infty}$.
\end{proof}

 \end{document}